\documentclass{amsart}
\usepackage[demo]{graphicx}
\usepackage{caption}
\usepackage{subcaption}
\usepackage{graphicx, amsmath, amsfonts, epsfig, latexsym, amssymb, amsthm}
\usepackage{tikz}
\usepackage{tkz-berge}

\usetikzlibrary {positioning}
\definecolor {processblue}{cmyk}{0.96,0,0,0}
\newtheorem{thm}{Theorem}[section]
\theoremstyle{definition}

\newtheorem{prop}[thm]{Proposition}
\newtheorem{defn}[thm]{Definition}
\newtheorem{lem}[thm]{Lemma}

\newtheorem{rem}[thm]{Remark}
\newtheorem{ex}[thm]{Example}

\numberwithin{equation}{section}
\begin{document}
\title[An ideal-based cozero-divisor graph of a commutative ring]
{An ideal-based cozero-divisor graph of a commutative ring}
\author{H. Ansari-Toroghy, F. Farshadifar, and F. Mahboobi-Abkenar}
\address{Department of pure Mathematics, Faculty of mathematical Sciences, University of Guilan,
P. O. Box 41335-19141, Rasht, Iran.}%
\email{ansari@guilan.ac.ir}%
\address{University of Farhangian, Tehran, Iran.}%
\email{f.farshadifar@gmail.com}%
\address{Department of pure Mathematics, Faculty of mathematical Sciences, University of Guilan, P. O. Box41335-19141, Rasht, Iran.}
\email{mahboobi@phd.guilan.ac.ir}%

\begin{abstract}
Let $R$ be a commutative ring and let $I$ be an ideal of $R$. In this paper, we introduce the cozero-divisor graph $\acute{\Gamma}_I(R)$ of $R$ and obtain some related results.
\end{abstract}
\maketitle
\section{Introduction}
\noindent
Let $R$ be a commutative ring with non-zero identity and let $Z(R)$ be the set of all zero-divisors of $R$. Anderson and Livingston, in \cite{2}, introduced the \emph{zero-divisor graph of R}, denoted by $\Gamma(R)$, as the (undirected) graph with vertices $Z^*(R) = Z(R)\backslash \{0\}$ and for two distinct elements $x$ and $y$ in $Z^*(R)$, the vertices $x$ and $y$ are adjacent if and only if $xy = 0$. 

In \cite{3}, Redmond introduced the definition of the zero-divisor graph with respect to an ideal. Let $I$ be an ideal of $R$. The \emph{zero-divisor graph of $R$ with respect to $I$}, denoted by $\Gamma_I(R)$, is the graph whose vertices are the set $\{ x \in R\setminus I\, |\, xy \in I\: for\: some\: y \in R \setminus I\}$ with distinct vertices $x$ and $y$ are adjacent if and only if $xy \in I$. Thus if $I = 0$, then $\Gamma_I(R) = \Gamma(R)$, and $I$ is a non-zero prime ideal of $R$ if and only if $\Gamma_I(R) = \emptyset$.

In \cite{1}, Afkhami and Khashayarmanesh introduced the \emph{cozero-divisor graph} $\acute{\Gamma}(R)$ of $R$, in which the vertices are precisely the nonzero, non-unit elements of $R$, denoted by $W^*(R)$, and two distinct vertices $x$ and $y$ are adjacent if and only if $x \not \in yR$ and $y \not \in xR$.

Let $I$ be an ideal of $R$. In this article, we introduce and study the cozero-divisor graph $\acute{\Gamma}_I(R)$ of $R$ with vertices $\{x \in R  \setminus Ann_R(I)\: |\: xI \neq I \}$ and two distinct vertices $x$ and $y$ are adjacent if and only if $x \not \in yI$ and $y \not \in xI$. This is a generalization of cozero-divisor graph introduced in \cite{1} when $I = R$, i.e, we have $\acute{\Gamma}_R(R) = \acute{\Gamma}(R)$. Moreover, this can be regarded the dual notion of ideal-based zero-divisor graph introduced by S.P. Redmond in \cite{3}. Also we obtain some useful information about secondal ideals.

We denote the set of maximal ideals and the Jacabson radical of $R$ by $Max(R)$ and $J(R)$, respectively. In a graph $G$, the distance between two distinct vertices $a$ and $b$, denoted by $d(a,b)$ is the length of the shortest path connecting $a$ and $b$. If there is not a path between $a$ and $b$, $d(a,b) = \infty$. The \emph{diameter} of a graph $G$ is $diam(G) = sup \{d(a,b) : a\:and\: b\:are\:distinct\:vertices\:of\: G\}$. The \emph{girth} of $G$, is the length of the shortest cycle in $G$ and it is denoted by $g(G)$. If $G$ has no cycle, we define the girth of $G$ to be infinite. An \emph{r-partite graph} is one whose vertex set can be partitioned into $r$ subsets such that no edge has both ends in any one subset. A \emph{complete r-partite graph} is one each vertex is jointed to every vertex that is not in the same subset. The complete bipartite (i.e, 2-partite) graph with part sizes $m$ and $n$ is denoted by $K_{m,n}$.

\section{On the generalization of the cozero-divisor graph}
\noindent

\begin{defn}\label{2.1}
Let $R$ be a commutative ring and let $I$ be an ideal of $R$. We define \emph{cozero-divisor graph $\acute{\Gamma}_I(R)$} of $R$ with vertices $\{x \in R\setminus Ann_R(I)\: |\: xI \neq I \}$. The distinct vertices $x$ and $y$ are adjacent if and only if $x \not \in yI$ and $y \not \in xI$. Clearly, when $I = R$ we have $\acute{\Gamma}_I(R) = \acute{\Gamma}(R)$.
\end{defn}

\begin{ex}\label{2.2}
Let $R = \Bbb Z _{12}$ and $I = (\bar{3})$. Then $\Gamma_I(R)=\emptyset$. Also, in the following figures we can see the deference between the graphs $\acute{\Gamma}(R)$, $\acute{\Gamma}_I(R)$, and $\Gamma(R)$.
\begin{figure}[h!]
\centering
\caption{}
\begin{subfigure}{0.2\textwidth}
\centering
\caption{$\acute{\Gamma}(R)$.}
\begin{tikzpicture}[auto,node distance=1.2cm,
  thick,main node/.style={circle,fill=black!10,font=\sffamily\tiny\bfseries}]
\node[main node] (1) {$\bar{8}$};
\node[main node] (2) [below of=1] {$\bar{6}$};
\node[main node] (3) [below of=2] {$\bar{4}$};
\node[main node] (4) [below right of=3] {$\bar{9}$};
\node[main node] (5) [below left of=3] {$\bar{3}$};
\node[main node] (6) [below right of=5] {$\bar{10}$};
\node[main node] (7) [below of=6] {$\bar{2}$};
  \path[every node/.style={font=\sffamily\small}]
    (1) edge node [left] {} (2)
    (1) edge node [left] {} (5)
    (1) edge node [left] {} (4)
    (2) edge node [left] {} (3)
    (3) edge node [left] {} (5)
    (3) edge node [left] {} (4)
    (5) edge node [left] {} (6)
    (5) edge node [left] {} (7)
    (4) edge node [left] {} (6)
    (4) edge node [left] {} (7);
\end{tikzpicture}
\end{subfigure}
\begin{subfigure}{0.2\textwidth}
\centering
\caption{$\acute{\Gamma}_I(R)$.}
\begin{tikzpicture}[auto,node distance=1.2cm,
  thick,main node/.style={circle,fill=black!10,font=\sffamily\tiny\bfseries}]
\node[main node] (1) {$\bar{2}$};
\node[main node] (2) [below of=1] {$\bar{10}$};
\node[main node] (3) [below of=2] {$\bar{6}$};
  \path[every node/.style={font=\sffamily\small}]
    (1) edge node [left] {} (2);
\end{tikzpicture}
\end{subfigure}
\begin{subfigure}{0.2\textwidth}
\centering
\caption{$\Gamma(R)$.}
\begin{tikzpicture}[auto,node distance=1.2cm,
  thick,main node/.style={circle,fill=black!10,font=\sffamily\tiny\bfseries}]
\node[main node] (1) {$\bar{8}$};
\node[main node] (2) [below of=1] {$\bar{6}$};
\node[main node] (3) [left of=2] {$\bar{10}$};
\node[main node] (4) [left of=3] {$\bar{3}$};
\node[main node] (5) [right of=2] {$\bar{2}$};
\node[main node] (6) [right of=5] {$\bar{9}$};
\node[main node] (7) [below of=2] {$\bar{4}$};
  \path[every node/.style={font=\sffamily\small}]
    (1) edge node [left] {} (2)
    (1) edge node [left] {} (6)
    (1) edge node [left] {} (4)
    (2) edge node [left] {} (3)
    (2) edge node [left] {} (5)
    (2) edge node [left] {} (7)
    (7) edge node [left] {} (6)
    (7) edge node [left] {} (4);
\end{tikzpicture}
\end{subfigure}
\end{figure}
\end{ex}

Let $I$ be an ideal of $R$. Then $I$ is said to be a \emph{second ideal} if $I \neq 0$ and for every element $r$ of $R$ we have either $rI= 0$ or $rI= I$.

\begin{lem}\label{second}
$I$ is a second ideal of $R$ if and only if $\acute\Gamma_I(R) = \emptyset$.
\begin{proof}
Straightforward.
\end{proof}

\end{lem}
\begin{thm}\label{2.3}
 Let $I$ be a proper ideal of $R$. Then we have the following.
 \begin{itemize}
   \item [(a)] The graph $\acute{\Gamma}_I(R)\setminus J(R)$ is connected.
   \item [(b)] If $R$ is a non-local ring, then diam $(\acute{\Gamma}_I(R)\setminus J(R)) \leq 2$.
 \end{itemize}
 \end{thm}
\begin{proof}
(a) If $R$ has only one maximal ideal, then $V(\acute{\Gamma}_I(R)) \setminus J(R)$ is the empty set; which is connected. So we may assume that $|Max(R)| > 1$. Let $a,b \in V(\acute{\Gamma}_I(R)) \setminus J(R)$ be two distinct elements. Without loss of generality, we may assume that $a \in bI$. Since $a \not\in J(R)$, there exists a maximal ideal $m$ such that $a \not\in m$. We claim that $m \nsubseteq J(R)\cup bI$. Otherwise, $m \subseteq J(R) \cup bI$. This implies that $m \subseteq J(R)$ or $m \subseteq bI$. But $m \neq J(R)$. Hence we have $m  \subseteq bI \varsubsetneq R$, so $m = bI$. This implies that $a \in m$, a contradiction. Choose the element $c \in m\setminus J(R) \cup bI$. It is easy to check that $a-c-b$.

(b) This follows from part (a).
\end{proof}

\begin{rem}
Figure (B) in Example \ref{2.2} shows that $J(R)$ can not be omitted in Theorem \ref{2.3}.
\end{rem}

\begin{thm}\label{2.4}
Let $R$ be a non-local ring and $I$ a proper ideal of $R$ such that for every element $a\in J(R)$, there exists $m \in Max(R)$ and $b \in m\setminus J(R)$ with $a \not\in bR$. Then $\acute{\Gamma}_I(R)$ is connected and diam$(\acute{\Gamma}_I(R)) \leqslant 3$.
\end{thm}
\begin{proof}
Use the technique of \cite[Theorem 2.5]{1}.
 \end{proof}

\begin{thm}\label{2.5}
Let $R$ be a non-local ring and $I$ be a proper ideal of $R$. Then $g(\acute{\Gamma}_I(R) \setminus J(R)) \leqslant 5$ or $g(\acute{\Gamma}_I(R) \setminus J(R)) = \infty$.
\end{thm}
\begin{proof}
Use the technique of \cite[Theorem 2.8]{1} along with Theorem \ref{2.3}.
\end{proof}

\begin{thm}\label{2.19}
Let $I$ be a non-zero ideal of a commutative ring $R$. If $V(\acute{\Gamma}(R)) \subseteq V(\acute{\Gamma}_I(R))$, then  $Ann_R(I)=0$ or $R$ is a field (so $I=R$ ). The converse holds if $I$ is finitely generated.
\end{thm}
\begin{proof}
Let $W^*(R)=V(\acute{\Gamma}(R)) \subseteq V(\acute{\Gamma}_I(R))$ and $Ann_R(I)\not=0$. Then $W^*(R)\subseteq R\setminus Ann_R(I)$. Thus $W(R) \cap Ann_R(I)=\{0\}$. Now suppose contrary that $I \not = R$. Let $0 \not = x \in Ann_R(I)$ and $y \in W(R)$. Then $xy \in W(R) \cap Ann_R(I)=\{0\}$ and $x  \not \in W(R)$. It follows that $y=0$ and hence $W(R)=\{0\}$. Therefore $R$ is a field. Conversely, if $I=R$ the result is clear. Now suppose that $I$ is a finitely generated ideal of $R$ such that $Ann_R(I)=0$ and $x \in V(\acute{\Gamma}(R))$. Then $xI \not=0$. If $xI=I$, then since $I$ is finitely generated, there exists $t \in R$ such that $(1+tx)I = 0$ by \cite[Theoram 75]{4}. Thus $1+tx \in Ann_R(I)=0$. This implies that $Rx=R$, which is a contradiction. Hence $x \in V(\acute{\Gamma}_I(R))$. Therefore $V(\acute{\Gamma}(R)) \subseteq V(\acute{\Gamma}_I(R))$.
\end{proof}

We will use the following lemma frequently in the sequel.

\begin{lem}\label{2.6}
Let $I \neq R$ be a finitely generated ideal of $R$ with $Ann_R(I) = 0$. Then $\acute\Gamma(R)$ is a subgraph of $\acute\Gamma_I(R)$.
\end{lem}
\begin{proof}
By Theorem \ref{2.19} we have $V(\acute\Gamma_I(R)) = V(\acute\Gamma(R))$. Now let $x, y \in V(\acute{\Gamma}(R)) = V(\acute{\Gamma}_I(R))$ and $x$ is adjacent to $y$ in $\acute{\Gamma}(R)$. Then clearly, they are adjacent in $\acute{\Gamma}_I(R)$. Otherwise, we may assume that $x \in yI$. This implies that $x \in yR$, which is a contradiction. Hence $\acute\Gamma(R)$ is a subgraph of $\acute\Gamma_I(R)$.
\end{proof}

The following example shows that the inclusion relation between $\acute\Gamma_I(R)$ and $\acute{\Gamma}(R)$ in Lemma \ref{2.6} may be a restrict inclusion.
\begin{ex}
Let $R := \Bbb Z$ and $I: = 5\Bbb Z$. Then $V(\acute{\Gamma}_I(R)) = V(\acute{\Gamma}(R)) = \Bbb Z \setminus \{-1,0,1\}$. Now by Lemma \ref{2.6}, $\acute\Gamma(R)$ is subgraph of $\acute\Gamma_I(R)$. However, the elements $2$ and $6$ are adjacent in $\acute\Gamma_I(R)$ but they are not adjacent in $\acute\Gamma(R)$.
\end{ex}

\begin{thm}\label{2.7}
Let $I \neq R$ be a finitely generated ideal of $R$ with $Ann_R(I) = 0$. Suppose that $|Max (R)|\geq 3$. Then $g(\acute\Gamma_I(R)) = 3$.
\end{thm}
\begin{proof}
Use the technique of \cite[Theorem 2.9]{1}.
\end{proof}

As we mentioned before, $V(\Gamma_I(R)) = \{ x \in R\setminus I\, |\, xy \in I\: for\: some\: y \in R \setminus I\}$. We will show this set by $Z_I(R)$. Clearly, for $I = 0$, $Z_I(R) = Z^*(R)$.
\begin{lem}\label{2.8}
Let $I \neq R$ be a finitely generated ideal of $R$ with $Ann_R(I) = 0$. Then $Z_I(R) \subseteq V(\acute{\Gamma}_I(R))$.
\end{lem}
\begin{proof}
If $I = 0$, then the claim is clear. So we assume that $I\neq 0$. Now let $x \in Z_I(R)$ then $x \neq 0$ and there exists $y \in R\setminus I$ such that $xy \in I$. Clearly, $xI \neq 0$. Further $xI \neq I$. Otherwise, $xI = I$. Since $I$ is finitely generated, there exists $t \in R$ such that $(1 + tx)I = 0$ by \cite[Theorem 75] {4}. This implies that $1 + tx = 0$. So $x$ is a unit element of $R$ and hence $y \in I$, which is a contradiction.  Therefore $x \in V(\acute{\Gamma}_I(R))$.
\end{proof}
The next example shows that the inclusion in Lemma \ref{2.8} is not strict in general.
\begin{ex}\label{2.8.1}
Let $I$ be a finitely generated ideal of $R$ with $Ann_R(I) = 0$. Further we assume that $R$ is an Artinian ring with $Z(R) \cap I = 0$. Then we have $V(\acute{\Gamma}_I(R)) = Z_I(R)$. To see this, it is enough to prove that $V(\acute{\Gamma}_I(R)) \subseteq Z_I(R)$ by Lemma \ref {2.8}. Let $x \in V(\acute{\Gamma}_I(R)$. Then we have $x \neq 0$ and $xI \neq I$. This implies that $xR \neq R$ and hence $x$ is a non-unit element of $R$. Since $R$ is Artinian, the set of non-unit elements of $R$ is the same as the set of zero-divisors of $R$. So $x \in Z(R)$. This shows that $x\not \in I$ and there exists $0 \neq y \in R\setminus I$ such that $xy = 0 \in I$. Clearly, $x, y \in Z(R)$. Therefore, $V(\acute{\Gamma}_I(R)) \subseteq Z_I(R)$.
\end{ex}
\begin{thm}\label{2.8.2}
Let $I$ be a finitely generated ideal of $R$ with $\sqrt{I} = I$ and $Ann_R(I) = 0$. Suppose that $Z_I(R) = V(\acute{\Gamma}_I(R))$. If $\Gamma_I(R)$ is complete, then $\acute{\Gamma}_I(R)$ is also a complete graph.
\end{thm}
\begin{proof}
Assume on the contrary that $\acute{\Gamma}_I(R)$ is not complete. So there exist $a, b \in V(\acute{\Gamma}_I(R))$ such that $a \in bI$ or $b \in aI$. Without loss of generality, we may assume that $a \in bI$. So, there exists $i \in I$ such that $a = bi$. We claim that $i$ is a unit element. Otherwise, $i \in V(\acute{\Gamma}(R)$. Thus we have $i \in V(\acute{\Gamma}_I(R))$ by Lemma \ref{2.6}. Hence $i \in Z_I(R)$ by assumption, which is a contradiction. Now $ab = b^{2}i \in I$. So there exist $i_1 \in I$ such that $b^{2}i = i_1$. Then $b^{2}= i^{-1}i_1 \in I$. Therefore, $b \in \sqrt{I} = I$, a contradiction.
\end{proof}

\begin{prop}\label{2.9}
Let $I$ be a proper ideal of $R$ and $\acute\Gamma_I(R)$ a complete bipartite graph with parts $V_i$, $i = 1,2$. Then every cyclic ideals $\textbf{a} ,\textbf{ b} \subseteq V_i$, for some $i =1,2$, are totally ordered.
\end{prop}
\begin{proof}
Assume on the contrary that there exist ideals $aR$ and $bR$ in $V_1$ such that $aR \nsubseteq bR$ and $bR \nsubseteq aR$. It follows that $b \not \in aR$ and $a \not \in bR$. Hence $b \not \in aI$ and $a \not \in bI$. This means $a$ is adjacent to $b$, a contradiction.
\end{proof}

\begin{prop}\label{2.10}
Let $I \neq R$ be a finitely generated ideal of $R$ with $Ann_R(I) = 0$. If the graph $\acute\Gamma_I(R)\setminus J(R)$  is n-partite for some positive integer n, then $|Max(R)| \leq n$.
\end{prop}
\begin{proof}
Assume contrary that $|Max(R)|> n$. Since  $\acute\Gamma_I(R)\setminus J(R)$ is a n-partite graph and $V(\acute\Gamma_I(R)) = V(\acute\Gamma(R))$ by Lemma \ref{2.6}, there exist $m, \acute{m} \in Max(R)$ and $a \in m \setminus \acute{m},b \in \acute{m} \setminus m$ such that $a, b$ belong to a same part. Clearly, $a \not \in bI$ and $b \not \in aI$, which is a contradiction.
\end{proof}

For a graph $G$, let $\chi (G)$ denote the \emph{chromatic number} of the graph
$G$, i.e., the minimal number of colors which can be assigned to the vertices of
$G$ in such a way that every two adjacent vertices have different colors. A \emph{clique}
of a graph $G$ is a complete subgraph of $G$ and the number of vertices in the largest
clique of $G$, denoted by $clique(G)$, is called the \emph{clique number} of $G$.
\begin{thm}\label{2.11}
 \item [(1)] Let $I \neq R$ be a finitely generated ideal of $R$ with $Ann_R(I) = 0$. Then if $R$ has infinite member of maximal ideal, then clique $\acute{\Gamma }_I(R)$ is also infinite; otherwise clique $(\acute{\Gamma }_I(R)) \geqslant |Max(R)|$.
 \item [(2)] If $\chi(\acute{\Gamma }_I(R)) < \infty$, then $|Max(R)| < \infty$.
 \end{thm}
\begin{proof}
(1) This follows from Lemma \ref{2.6} and \cite[Theorem 2.14]{1}.

(2) Use part (1) along with \cite[Theorem 2.14]{1}.
\end{proof}

\begin{thm}\label{2.12}
Let $R = S_1 + S_2$, where $S_1$ and $S_2$ are second ideals of $R$. If $P_1 = Ann_R(S_1)$ and $P_2 = Ann_R(S_2)$, then $V(\acute{\Gamma}(R)) = (P_1 \setminus P_2) \cup (P_2 \setminus P_1)$ and $\acute{\Gamma }(R)$ is a complete bipartite graph.
\end{thm}
\begin{proof}
Let $x \in V(\acute{\Gamma}(R))$, so we have $xR \neq 0$ and $xR \neq R$. Since $xR \neq 0$, $xS_1 \neq 0$ or $xS_2 \neq 0$. First we show that $V(\acute{\Gamma}(R)) = (P_1 \setminus P_2) \cup (P_2 \setminus P_1)$. If $xS_1 \neq 0$, then $x \not \in P_1$. So  $xS_1 = S_1$. We claim that $xS_2 = 0$. Otherwise, $xS_2 \neq 0$ so that $x \not \in P_2$. It means that $xS_2 = S_2$. Thus $xR = R$, a contradiction. So we have $x \in P_2$ hence $x \in (P_2 \setminus P_1) \cup (P_2 \setminus P_1)$. We have similar arguments for reverse inclusion. Now let $x \in P_1 \setminus P_2$ and $y \in P_2 \setminus P_1$. We show that $x \not \in yR$ and $y \not \in xR$. Otherwise, $x \in yR$ or $y \in xR$. Without loss of generality, $x \in yR$. Then there exists $t \in R$ such that $x = ty$. But $x\not\in P_2$ implies that $ty \not \in Ann_R(S_2)$ so that  $tyS_2 \neq 0$ , a contradiction. Thus, $x$ is adjacent to $y$. Now we show that $x$ and $y$ can not lie in $P_1 \setminus P_2$ or $P_1 \setminus P_2$. To see this let $x, y \in P_1 \setminus P_2$ and assume that they are adjacent. Then we have $x \not \in yR$ and $y \not \in xR$. Now by using our assumptions, we conclude that $x \not \in xR$, a contradiction.
\end{proof}
\begin{thm}\label{2.12.1}
Let $I \neq R$ be a finitely generated ideal of $R$ with $Ann_R(I) = 0$. Assume that $|Max(R)| \geq 5$. Then $\acute{\Gamma}_I(R)$ is not planar.
\end{thm}
\begin{proof}
This follows from Lemma \ref{2.6} and \cite [Theorem 3.9]{1}.
\end{proof}

\begin{prop}\label{2.13}
Let $R$ be a ring and $I$ be a proper ideal. Then the following hold.
\begin{itemize}
  \item [(a)] $V(\Gamma_{Ann(I)}(R))\subseteq V(\acute{\Gamma}_I(R))$.
  \item [(b)] If $R$ be a reduced ring, then $\Gamma_{Ann(I)}(R)$ is a subgraph of $\acute{\Gamma}_I(R)$.
\end{itemize}
\end{prop}
\begin{proof}
 (a) Let $x \in V(\Gamma_{Ann(I)}(R))$. Then there  exists $y \in R\setminus Ann_R(I)$ such that $xy \in Ann_R(I)$. We claim that $xI \neq I$. Otherwise, $xI = I$. Then $xyI = yI$ so that $yI = 0$. This implies that $y \in Ann_R(I)$, a contradiction. Therefore, $V(\Gamma_{Ann(I)}(R))\subseteq V(\acute{\Gamma}_I(R)) $.

 (b) By part (a),  $V(\Gamma_{Ann(I)}(R))\subseteq V(\acute{\Gamma}_I(R))$. Now we suppose that $x$ is adjacent to $y$  in $\Gamma_{Ann(I)}(R)$. We show that $x$ is adjacent to $y$ in $\acute{\Gamma}_I(R)$. Otherwise, without loss generality, we assume that $x \in yI$. So that ${x}^2 \in xyI$. Thus ${x}^2 = 0$. This implies that $x \in Ann_R(I)$, a contradiction.
\end{proof}

\begin{prop}\label{2.131}
Let $I$ be a finite generated proper ideal of $R$. Suppose that $x,y \in R\setminus Ann_R(I)$.
\begin{itemize}
  \item [(a)] $x \in V((\acute{\Gamma}_I(R))$ if and only if $x+Ann_R(I) \in V(\acute{\Gamma}(R/Ann_R(I))$.
  \item [(b)] If $x+Ann_R(I) $ is adjacent to $y+Ann_R(I)$ in $\acute{\Gamma}(R/Ann_R(I))$, then $x$ is adjacent to $y$ in $\acute{\Gamma}_I(R)$.
\end{itemize}
\end{prop}
\begin{proof}
This is straightforward.
\end{proof}

An $R$-module $M$ is said to be a \emph{comultiplication module} if for every submodule $N$ of $M$ there exists an ideal $I$ of $R$ such that $N = Ann_M(I)$, equivalently, for each submodule $N$ of $M$, we have $N = Ann_M(Ann_R(N))$ \cite{AF07}. $R$ is said to be a \emph{comultiplication ring} if $R$ is a comultiplication $R$-module.
\begin{thm}\label{2.14}
Let $I$ be a proper ideal of $R$. Then $V(\acute{\Gamma}_I(R))=V(\Gamma_{Ann(I)}(R))$ if one of the following conditions hold.
\begin{itemize}
  \item [(a)] $R$ is a comultiplication ring.
  \item [(b)] $R/Ann_R(I)=Z(R/Ann_R(I)) \cup U(R/Ann_R(I))$.
\end{itemize}
\end{thm}
\begin{proof}
Clearly $V(\Gamma_{Ann(I)}(R))\subseteq V(\acute{\Gamma}_I(R))$.

(a) Let $x \in V(\acute{\Gamma}_I(R))$. Then $xI \not =0$ and $xI \not = I$. Since $R$ is a comultiplication ring, this implies that $Ann_R(xI) \not = Ann_R(I)$. Thus there exists $y \in Ann_R(xI) \setminus Ann_R(I)$. Therefore,  $x \in V(\Gamma_{Ann(I)}(R))$.

(b) Let $x \in V(\acute{\Gamma}_I(R))$. Then $xI \not =0$ and $xI \not = I$. By assumption, $x+Ann_R(I) \in Z(R/Ann_R(I))$ or $x+Ann_R(I) \in U(R/Ann_R(I))$. If $x+Ann_R(I) \in Z(R/Ann_R(I))$, then there exists $y \in R \setminus Ann_R(I)$ such that $xy \in Ann_R(I)$. Therefore, $x \in V(\Gamma_{Ann(I)}(R))$. If $x+Ann_R(I) \in U(R/Ann_R(I))$, then there exists $z+Ann_R(I) \in R/Ann_R(I)$ such that $xz+Ann_R(I)=1+Ann_R(I)$. Thus $1=xz+a$ for some $a \in Ann_R(I)$. Now we have $I=1I=(xz+a)I=xzI \subseteq xI$, a contradiction.
\end{proof}

\begin{thm}\label{2.15}
Let $I \subseteq J$ be non-zero ideals of $R$. Then we have the following.
\begin{itemize}
  \item [(a)] If $R/Ann_R(J)=Z(R/Ann_R(J)) \bigcup U(R/Ann_R(J))$, then $V(\acute{\Gamma}_I(R)) \subseteq V(\acute{\Gamma}_J(R))$.
  \item [(b)] If $dim(R)=0$, then $V(\acute{\Gamma}_I(R)) \subseteq V(\acute{\Gamma}_J(R))$. In particular, this holds if R is a finite
      commutative ring.
\end{itemize}
\end{thm}
\begin{proof}
(a) This follows from Theorem \ref{2.14} (b) and \cite[Theorem 2.8]{AS13}.

(b) $dim(R)=0$ implies that $dim (R/J)=0$. It follows that
 $$
 R/Ann_R(J)=Z(R/Ann_R(J)) \bigcup U(R/Ann_R(J)).
 $$
 Now the result follows from part (a).
\end{proof}

\begin{prop}\label{2.16}
Let $I$ be a non-zero ideal of a commutative ring $R$ with $R=Z(R) \cup U(R)$ and  $V(\acute{\Gamma}_I(R))=V(\acute{\Gamma}(R))$. Then $Ann_R(I) = 0$.
\end{prop}
\begin{proof}
Suppose that $V(\acute{\Gamma}_I(R))=V(\acute{\Gamma}(R))$. Since $V(\acute{\Gamma}_I(R)) \subseteq R\setminus Ann_R(I)$, we have $V(\acute{\Gamma}(R))\subseteq R\setminus Ann_R(I)$. Thus $Ann_R(I) \subseteq R \setminus V(\acute{\Gamma}(R))=\{0\} \cup U(R)$ by hypothesis. Therefore, $Ann_R(I) = 0$.
\end{proof}

\section{secondal ideals}
In this section, we will study the ideal-based cozero-divisor graph with respect to secondal ideals.

The element $a \in R$ is called \emph{prime to an ideal} $I$ of $R$ if $ra \in I$ (where $r \in R$) implies that $r \in I$.  The set of elements of $R$ which are not prime to $I$ is denoted by $S(I)$. A proper ideal $I$ of $R$ is said to be \emph{primal} if $S(I)$ is an ideal of $R$ \cite{8}.

A non-zero submodule $N$  of an $R$-module $M$ is said to be \emph{secondal} if $W_R(N) = \{a \in R\,:\,aN \neq N\}$ is an ideal of $R$ \cite{AF101}. A \emph{secondal ideal} is defined similarly when $N = I$ is an ideal of $R$. In this case, we say $I$ is $P$-\emph{secondal}, where $P = W(I)$ is a prime ideal of $R$.

\begin{lem}\label{3.2}
Let $I$ be a non-zero ideal of $R$. Then the following hold.
\begin{itemize}
 \item [(a)] $Ann_R(I) \subseteq W(I)$.
  \item [(b)] $Z_R(R / Ann_R(I)) \subseteq W(I)$.
 \item [(c)] $ V(\acute{\Gamma}_I(R)) = W(I) \setminus Ann_R(I)$. In particular, $ V(\acute{\Gamma}_I(R)) \cup Ann_R(I) = W(I)$.
  \item [(d)] If $Ann_R(I)$ is a radical ideal of $R$, then $\bigcup_{P\in Min(Ann_R(I))} P \subseteq W(I)$.
\end{itemize}
\end{lem}
\begin{proof}
(a) Let $r \in Ann_R(I)$. Then $rI = 0\neq I$. Thus $r \in W(I)$.

(b) Let $x \in Z_R(R / Ann_R(I))$ and $x \not \in W(I)$. Then there exists $y \in R \setminus Ann_R(I)$ such that $xyI = 0$. Hence $xI = I$ implies that $yI = 0$, a contradiction.

(c) Let $r \in V(\acute{\Gamma}_I(R))$. Then $r \in R \setminus Ann_R(I)$ and $rI \neq I$; hence $r \in W(I) \setminus Ann_R(I)$. Thus $V(\acute{\Gamma}_I(R)) \subseteq W(I) \setminus Ann_R(I)$. Conversely, we assume that $x \in W(I) \setminus Ann_R(I)$. So $xI \neq I$ and $xI \neq 0$. Then $x \in V(\acute{\Gamma}_I(R))$, so we have equality.

(d) By \cite [Exer 13, page 63]{4}, $Z_R (R / I) = \bigcup_{P\in Min(I)} P$, where $I$ is a radical ideal of $R$. Thus $Z_R (R / Ann_R(I)) =  \bigcup_{P\in Min(Ann_R(I))} P$. Hence $\bigcup_{P\in Min(Ann_R(I))} P \subseteq W(I)$ by part (b).
\end{proof}

\begin{rem}\label{3.3}
Let $R = \Bbb Z$, $I = 2\Bbb Z$. Then $Z_R( R / Ann_R(I)) = Z_R(R) = {0}$ and $W(I) = \Bbb Z \setminus \{-1 , 1\}$. Therefore the converse of part (b) of the above lemma is not true in general.
\end{rem}

\begin{prop}\label{3.4}
Let $I$ and $P$ be ideals of $R$ with $Ann_R(I) \subseteq P$. Then $I$ is a P-secondal ideal of $R$ if only if $V(\acute{\Gamma}_I(R)) = P \setminus Ann_R(I)$.
\end{prop}
\begin{proof}
Straightforward.
\end{proof}

\begin{thm}\label{3.5}
Let $I$ be an ideal of $R$. Then $I$ is a secondal ideal of $R$ if and only if $V(\acute{\Gamma}_I(R)) \cup Ann_R(I) $ is an (prime) ideal of $R$.
\end{thm}
\begin{proof}
Let $I$ be a secondal ideal. Then $W(I)$ is a prime ideal and by Lemma \ref{3.2}(c) $V(\acute{\Gamma}_I(R)) \cup Ann_R(I) = W(I)$. Thus $V(\acute{\Gamma}_I(R)) \cup Ann_R(I)$ is an ideal of $R$. Conversely, suppose that $V(\acute{\Gamma}_I(R)) \cup Ann_R(I)$ is a (prime) ideal. Then by Lemma \ref{3.2}(c) , $V(\acute{\Gamma}_I(R)) \cup Ann_R(I) = W(I)$ is a prime ideal. Hence $I$ is a secondal ideal.
\end{proof}

\begin{thm}\label{3.6}
Let $I$ and $J$ be P-secondal ideals of $R$. Then $V(\acute{\Gamma}_I(R)) = V(\acute{\Gamma}_J(R))$ if and only if $Ann_R(I) = Ann_R(J)$.
\end{thm}
\begin{proof}
By Lemma \ref{3.2} (a), $Ann_R(I) \subseteq P$ and $Ann_R(J) \subseteq P$. It then follows from Proposition \ref{3.4} that $V(\acute{\Gamma}_I(R)) = V(\acute{\Gamma}_J(R))$ if and only if $P \setminus Ann_R(I) = P \setminus Ann_R(J)$; and this holds if and only if $Ann_R(I) = Ann_R(J)$.
\end{proof}

\begin{lem}\label{3.7}
Let $N$ be a secondary submodule of an $R$-module $M$. Then $\sqrt{Ann_R(N)} = W(N)$.
\end{lem}
\begin{proof}
Let $x \in W(N)$. Then $xN \neq N$. Since $N$ is a secondary $R$-module, there exists a positive integer $n$ such that $x^{n}N = 0$. Thus $x \in \sqrt{Ann_R(N)} $. Hence $W(I) \subseteq \sqrt{Ann_R(I)}$. To see the reverse inclusion, let $x \in \sqrt{Ann_R(N)}$ and $x \not\in W(N)$. Then $x^{n}N = 0$ for some positive integer $n$ and $xN = N$. Therefore $N = 0$, a contradiction.
\end{proof}

\begin{thm}\label{3.8}
Let $I$ be an ideal of $R$. Then $I$ is secondary ideal if and only if $V(\acute{\Gamma}_I(R)) = \sqrt{Ann_R(I)} \setminus Ann_R(I)$.
\end{thm}
\begin{proof}
If $I$ is secondary, then $\sqrt{Ann_R(I)} = W(I)$ by Lemma \ref{3.7}. Hence $I$ is a $\sqrt{Ann_R(I)}$-secondal ideal of $R$. Then Proposition \ref{3.4} implies that $V(\acute{\Gamma}_I(R)) = \sqrt{Ann_R(I)} \setminus Ann_R(I)$. Conversely, suppose that $x \in R, xI \neq I$, and $x \not \in \sqrt{Ann_R(I)}$. Then $x \in W(I)$ and $x \not \in Ann_R(I)$. Thus $x \in V(\acute{\Gamma}_I(R))$ and so $x \in \sqrt{Ann_R(I)} \setminus Ann_R(I)$ by assumption, a contradiction.
\end{proof}

\begin{defn}\label{3.9}
Let $I$ be an ideal of $R$. We say that an ideal $J$ of $R$ is \emph{second} to $I$ if $IJ = I$.
\end{defn}

\begin{prop}\label{3.10}
Let $I$ be an ideal of $R$. If $I$ is not secondal, then there exist $x , y \in V(\acute{\Gamma}_I(R))$ such that $<x,y>$ is second to $I$.
\end{prop}
\begin{proof}
Suppose that $I$ is an ideal of $R$ such that it is not secondal. Then by Lemma \ref{3.2} (c), $V(\acute{\Gamma}_I(R)) \cup Ann_R(I) = W(I)$ is not an ideal of $R$, so there exist $x,y \in W(I)$ with $x - y \not \in W(I)$ and so $(x - y)I = I$. Hence $<x,y>I = I$. Now we claim that $x, y \not \in Ann_R(I)$. Otherwise, we have $x \in Ann_R(I)$ or $y \in Ann_R(I)$. If $x, y \in Ann_R(I)$, then $x-y \in Ann_R(I) \subseteq W(I)$, a contradiction. If $x \in Ann_R(I)$ and $y \not \in Ann_R(I)$, then $I = (x-y)I \subseteq xI+yI = 0+yI$, a contradiction. Similarly, we get a contradiction when $x\not \in Ann_R(I)$ and $y \in Ann_R(I)$. Thus we have $x, y \not \in Ann_R(I)$.
\end{proof}

\begin{prop}\label{3.11}
Let $I$ be an ideal of $R$. Then the following hold.
 \begin{itemize}
  \item [(a)] Let $x,y$ be distinct elements of $\sqrt{Ann_R(I)} \setminus Ann_R(I)$ with $xy \not\in Ann_R(I)$. Then the ideal $<x,y>$ is not second to $I$.
  \item [(b)] If $I$ is a secondary ideal, then the $diam (\Gamma_{Ann(I)}(R)) \leq 2$.
\end{itemize}
\end{prop}
\begin{proof}
(a) Let ideal $<x,y>$ be second to $I$. Since $x, y \in \sqrt{Ann_R(I)} \setminus Ann_R(I)$, there exists the least positive integer $n$ such that $x^ny \in Ann(I)$. As $xy \not \in Ann_R(I)$, we have $n \geqslant 2$. Let $m$ be the least positive such that $x^{n-1}y^m \in Ann_R(I)$. Now clearly $m \geqslant 2$ because $x^{n-1}y \not \in Ann_R(I)$. This yields that the contradiction
$$
0=x^{n-1}y^{m-1}(x, y)I = x^{n-1}y^{m-1}I\not =0.
$$

(b) If $I$ is secondary, then $W(I) = \sqrt{Ann_R(I)}$ by Lemma \ref{3.8}. Choose two distinct vertices $x, y$ in $\Gamma_{Ann(I)}(R)$. If $xy \in Ann_R(I)$, then $d(x, y) = 1$. So we assume that $xy \not \in Ann_R(I)$. Then by Proposition \ref{2.13} (a) and Lemma \ref{3.2}, $x, y \in W(I)\setminus Ann_R(I)$. Also we have $x, y \in \sqrt{Ann_R(I)}\setminus Ann_R(I)$ by Theorem \ref{3.8}. As in the proof of (a), we have the path $x- x^{n-1}y^{m-1}- y$ from $x$ to $y$ in $\Gamma_{Ann(I)}(R)$. Hence $d(x, y) = 2$. Therefore, $diam(\Gamma_{Ann(I)}(R)) \leq 2$.
\end{proof}

\bibliographystyle{amsplain}

\end{document}